\documentclass[a4paper,12pt]{amsart}
\usepackage{amsthm, amssymb, amscd}
\usepackage{mathrsfs, verbatim}
\newtheorem{thm}{Theorem}[section]
\newtheorem{prop}[thm]{Proposition}
\newtheorem{cor}[thm]{Corollary}
\newtheorem{lm}[thm]{Lemma}

\newtheorem{obs}[thm]{Observation}
\newtheorem{definition}[thm]{Definition}
\newtheorem{question}{Question}
\def\OO{\mathcal{O}}

\def\P{\mathcal{P}}

\def\C{\mathbb{C}}

\def\PP{\mathbb{P}}
\def\Z{\mathbb{Z}}

\input diagxy
\begin{document}
\title{Principal part bundles on $\PP^n$ and quiver representations}
\date{October 20, 2011}
\author{Riccardo Re}
\address{Dipartimento di Matematica e Informatica, Universit\'a di Catania, Italy}
\email{riccardo@dmi.unict.it}
\thanks{The author thanks the Department of Mathematics of the University of Milano for the warm hospitality during the preparation of this article.}
\begin{abstract}{We study the principal parts bundles $\P^k\OO_{\PP^n}(d)$ as homogeneous bundles and we describe their associated quiver representations.  With this technique we show that if $n\geq 2$ and $0\leq d<k$ then there exists an invariant splitting $\P^k\OO_{\PP^n}(d)=Q_{k,d}\oplus (S^dV\otimes \OO_{\PP^n})$ with $Q_{k,d}$ a stable homogeneous vector bundle. The splitting properties of such bundles were previously known only for $n=1$ or $k\leq d$ or $d<0$. Moreover we show that the Taylor truncation maps    $H^0\P^k\OO_{\PP^n}(d)\to H^0\P^h\OO_{\PP^n}(d)$, defined for any $h\leq k$ and any $d$, have maximal rank.}\end{abstract}
\maketitle
\section{Introduction}
The principal part sheaves $\P^k(\mathcal{F})$  were defined by Grothendieck in \cite{groth} for any relative scheme $X\to S$ and any $\OO_X$-module $\mathcal{F}$. In the context of differential geometry and $\mathcal{F}$ a vector bundle on $X$, they are also known as the {\em bundles of k-jets} of sections of $\mathcal{F}$. Grothendieck uses p.p. sheaves in his definition of the sheaf of 1-differentials $\Omega_{X/S}$ and to define linear differential operators between $\OO_X$-modules and introduce the universal property of p.p. sheaves with respect to differential operators.  

It is not trivial to understand the structure of p.p. sheaves already in the case of $X=\PP(V)$, a projective space over an algebraically closed field. In this context the simplest sheaves $\mathcal{F}$ to consider are the line bundles $\OO_{\PP^n}(d)$. In the article \cite{som} projective spaces and abelian varieties were shown to be the only projective varieties for which there exists a line bundle $L$ such that $\P^k(L)=M^{\oplus r}$ for some $k$, with $M$ another line bundle. In \cite{perk} and \cite{maake1}, \cite{maake2} one can find many results on the structure of $\P^k (L)$ for $L$ a line bundle on a toric variety or a homogeneous space, respectively. In the cited papers the splitting type of $\P^k\OO_{\PP^n}(d)$ for $k\leq d$ or $d<0$ is calculated. For example it is shown that $\P^k\OO_{\PP^n}(d)\cong S^kV\otimes\OO_{\PP^n}(d-k)$ for $d\geq k$. On the other hand we have not been able to find any reference for the splitting type of $\P^k\OO_{\PP^n}(d)$ in the cases when $n\geq 2$ and $k>d\geq 0$ in the existing literature. The aim of this paper is to fill that gap, by proving the following theorem.
\begin{thm}\label{thm:main1} If $n\geq 2$ and $k>d\geq 0$ one has $\P^k\OO_{\PP^n}(d)\cong Q_{k,d}\oplus (S^dV\otimes\OO_{\PP^n})$, with $Q_{k,d}$ a stable homogeneous vector bundle on $\PP^n$. Moreover, for any $k>d\geq 0$, $Q_{k,d}$ admits a $Sl(n+1)$-invariant filtration 
$0=F_{k+1}\subset F_k\subset\cdots\subset F_{d+1}=Q_{k,d}$ with quotients $F_{i}/F_{i+1}=S^i\Omega_{\PP^n}(d)$.\end{thm}
We are also interested in understanding the Taylor truncation maps $\tau_{k,h}:\P^k\OO_{\PP^n}(d)\to\P^h\OO_{\PP^n}(d)$, with $k\geq h$, in view of possible application to differential algebra questions. In this context we get the following result.
\begin{thm}\label{thm:main2}  For any $k\geq h$ and $d$ the Taylor truncation map induces a linear map
$\tau: H^0\P^k\OO_{\PP^n}(d)\to H^0\P^h\OO_{\PP^n}(d)$ with maximal rank.
\end{thm}
Besides the interest in completing the previously known results on p.p. bundles on projective spaces, our motivation comes also from the hope that p.p. bundles of type $\P^k\OO_X(n)$ with $k>1$ and $d\geq 0$ on a smooth projective variety $X$ might be useful in the study of the embedding $X\subset \PP^N$. We recall that the canonical surjective morphism $\P^1\OO_{\PP(V)}(1)|_X=V\otimes \OO_X\to \P^1\OO_X(1)$ gives rise to a map $\PP(\P^1\OO_X(1))\to \PP(V)$ whose image is the tangential variety Tan$(X)$, i.e. the union of the projective tangent spaces of $X$. This observation has been considered in the literature, implicitly or explicitly, for example in the articles \cite{ein}, \cite{bodeoli}, in the study of properties of the embedding of $X$, and in the monograph \cite{gzk}, in the definition of the discriminants of algebraic varieties. In our opinion it is possible that interesting geometrical interpretations and applications might arise also from the higher p.p. bundles $\P^k\OO_X(d)$ with $d\geq 1$, in the near future.

In section \ref{sec:pp} we recall those properties of p.p. sheaves that we will need later and in section \ref{sec:homogeneous} we give a survey of the theory of quiver representations associated with homogeneous vector bundles on projective spaces, giving also proofs to some useful technical results which are not easy to find in the existing literature. We provide the proofs of Theorems \ref{thm:main1} and \ref{thm:main2} in sections \ref{sec:thm1} and \ref{sec:thm2}, respectively.
\subsubsection*{Acknowledgements} I wish to express my gratitude to Alberto Alzati for his hospitality at the Department of Mathematics of the University of Milan and for many nice mathematical discussions during the preparation of this paper. I also thank Giorgio Ottaviani for his useful advise on many occasions, his encouragement and for introducing me to the theory of quiver representations for homogeneous vector bundles. 
\section{Principal parts sheaves.}\label{sec:pp}
Throughout this article we will denote $\PP(V)$ a $n$-dimensional projective space over the complex numbers $\C$, and this will be the only algebraic variety we have to say something new about. Despite that, we find it useful to start in great generality, by recalling the definition given in \cite{groth} of the p.p. sheaves associated to a $\OO_X$-module on an arbitrary scheme $X$ over a base scheme $S$. 
We denote $\mathcal{F}$ a $\OO_X$-module, $X\times_S X$ the fibred product over $S$,  $p_i:X\times X\to X$ for $i=1,2$ the two projections and $\Delta\subset X\times_S X$ the diagonal. We also denote  $\Delta_{k+1}$ the multiple structure on the diagonal $\Delta\subset X\times_S X$ defined by the $(k+1)$-th power $\mathcal{I}_\Delta^{k+1}$ of the ideal sheaf $\mathcal{I}_\Delta\subset\OO_{X\times_S X}$.  Then $\OO_{\Delta_{k+1}}$ has two different structures as a sheaf of $\OO_X$-algebras, coming from the maps $p_i^\ast:\OO_X\to {p_i}_\ast\OO_{\Delta_{k+1}}$ for $i=1,2$, which we will refer to as the {\em first} and the {\em second} $O_X$-algebra structure of $\OO_{\Delta_{k+1}}$, respectively. We recall that the principal parts sheaves of  $\mathcal{F}$ are defined in the following way:
 $$\P^k(\mathcal{F})={p_1}_\ast(p_2^\ast\mathcal{F}|_{\Delta_{k+1}}).$$
 This construction defines a covariant functor of the category of $\OO_X$-modules to itself.
 With a slight abuse of notation one can say that $\P^k(\mathcal{F})$ is identified with the sheaf $\OO_{\Delta_{k+1}}\otimes_2\mathcal{F}$, with the structure of $\OO_X$-module arising from the first $\OO_X$-structure on the factor $\OO_{\Delta_{k+1}}$. 
 
When $X$ is a smooth algebraic variety one can think at $\P^k(\mathcal{F})$ as the sheaf with stalks given by all the Taylor polynomials up to the $k$-th order of sections of $\mathcal{F}$ defined at $x_0\in X$, with respect to a given a set of parameters for $x_1,\ldots,x_n$ centered at $x_0$. In general there exists a canonical sheaf map (but not a map of $\OO_X$-modules) 
\begin{equation}\label{eq:d} d_k: \mathcal{F}\to \P^k(\mathcal{F}),\end{equation} mapping a local section defined at a point $x$ to its $k$-th Taylor polynomial at $x$. This map is defined by $f\mapsto (1\otimes_{\OO_S} f)\otimes_{\OO_{X\times_{S}X}}1\in\P^k(\mathcal{F})$, for any $f\in\mathcal{F}$.  As a consequence, the image of $d_k$ generates $\P^k(\mathcal{F})$ as a $\OO_X$-module. For example, assume $X$ smooth with local parameters $x_1,\ldots,x_n$ as above and $\mathcal{F}$ a locally free sheaf with local basis $e_1,\ldots, e_r$. Then a local section $\sigma(x)=\sum f(x)_ie_i$ of $\mathcal{F}$ produces \begin{equation}\label{eq:dk}d_k(\sigma)=\sum f_i(x+h)e_i \mbox{ mod }(h)^{k+1}\end{equation}
In particular one sees that a generating set for the fiber $\P^k\mathcal{F}(x_0)$ is given by
\begin{equation}\label{eq:basePk} \{h_1^{j_1}\cdots h_n^{j_n}e_i,\ :\ i=1,\ldots,n,\ j_1+\cdots+j_n\leq k\}\end{equation}

 When $\mathcal{F}$ is a locally free sheaf, there exist canonical exact sequences of $\OO_X$-modules:
 \begin{equation}\label{eq:canonicseqpp} 0\to S^k\Omega_X^1\otimes\mathcal{F}\to \P^k(\mathcal{F})\to \P^{k-1}(\mathcal{F})\to 0,\end{equation}
 arising from the exact sequences $$0\to \mathcal{I}_{\Delta}^k/\mathcal{I}_{\Delta}^{k+1}\to \OO_{\Delta_{k+1}}\to\OO_{\Delta_k}\to 0.$$ 
 
 The restriction $p_2^\ast\mathcal{F}|_{\Delta_{k+1}}\to p_2^\ast\mathcal{F}|_{\Delta}\cong\mathcal{F}$ induces the surjection $\P^k\mathcal{F}\to \mathcal{F}$. Composing the map $d_k: \mathcal{F}\to \P^k(\mathcal{F})$ with the surjection above, one finds that the linear map \begin{equation}\label{eq:surjsec}H^0\P^k(\mathcal{F})\to H^0(\mathcal{F})\end{equation} is surjective.  
We conclude this short introduction to p.p. sheaves by mentioning their universal property with respect to differential operators.
We recall that a differential operator of order $0$ between two $\OO_X$-modules $D:\mathcal{F}\to \mathcal{G}$ is a $\OO_X$-homomorphism and that differential operators of order $\leq k$ are defined inductively as those sheaf maps $D:\mathcal{F}\to \mathcal{G}$ such that for any local section $a\in \OO_X$ the operator $D_a:\mathcal{F}\to\mathcal{G}$ defined by  $D_a(f)=D(af)-aD(f)$, with $f\in\mathcal{F}$ a local section,  is a differential operator of order $\leq k-1$. Then the following property holds.
\begin{prop}[cfr. \cite{groth} Def. 16.8.1 and Prop. 16.8.8]\label{prop:univpp} $D:\mathcal{F}\to\mathcal{G}$ is a differential operator of order $\leq k$ iff there exists a morphism of $\OO_X$-modules $\theta:\P^k(\mathcal{F})\to \mathcal{G}$ such that the following diagram is commutative
$$\bfig
\btriangle[{\mathcal{F}}`{\P^k(\mathcal{F})}`{\mathcal{G}};{d_k}`{D}`{\theta}]
\efig .$$
Moreover, the morphism $\theta$ is uniquely determined by $D$.
\end{prop}
\section{Homogeneous vector bundles on projective spaces.}\label{sec:homogeneous}
In this article we will deal with the case $X=\PP(V)$, with $V\cong\C^{n+1}$ and $\mathcal{F}=\OO_{\PP(V)}(d)$ a line bundle, which we will denote simply by $\OO(d)$. Recall that $\PP(V)$ is the homogeneous space $\PP(V)=G/P$ with $G=SL(n+1)$ and $P$ the parabolic subgroup consisting of matrices of the form
\begin{equation}\label{eq:formP}p=\left(\begin{array}{ll} \lambda&\underline{0}\\ \underline{v} & A\end{array}\right)\end{equation} with $\underline{v}\in\C^n$, $A\in Gl(n)$ and $\lambda=\det(A)^{-1}$, acting on $V$ on the left and fixing the set of hyperplanes which contain $x=(1:0:\cdots:0)\in\PP^n$. 
Since $\OO(d)$ is a homogeneous line bundle, by functoriality it follows that $\P^k\OO(d)$ are homogeneous for any $k$ and $d$. Moreover the following exact sequences, deduced from (\ref{eq:canonicseqpp}), are built with $G$-invariant homomorphisms
\begin{equation}\label{eq:exactk} 0\to S^k\Omega(d)\to \P^k\OO(d)\to \P^{k-1}\OO(d)\to 0.\end{equation}
In the study of the structure of $\P^k\OO(d)$ as a homogeneous vector bundle we will need some related notions of representation theory, so we will now enter into this topic in some detail.

\subsection*{$P$-modules and vector bundles.}
Let $X=G/P$ be a rational homogeneous variety, with $G$ semisimple and $P$ parabolic, and let $x\in X$ be the point associated with the lateral $P\in G/P$. A $G$-homogeneous vector bundle $E$ on $X$ with fiber $W=E(x)$ over $x$ is associated with a representation $\rho: P\to GL(W)$ in the following way: $E$ is the quotient of the product $G\times W$ by the relation $(g,v)\sim (g',v')$ iff there exists $p\in P$ such that $g'=gp$ and $v'=\rho(p^{-1})(v)$. One denotes then $E=E_{\rho}$.  As an immediate application of the correspondence between $P$-modules and homogeneous vector bundles one sees for example that the line bundle $\OO(1)$ is given by $\rho(p)=\lambda\in Gl(1)$, for any $p\in P$ of the form (\ref{eq:formP}).

The Lie algebra $\mathcal{P}$ of $P$ is the direct sum $\mathcal{P}=Lie R\oplus Lie N=\mathcal{R}\oplus\mathcal{N}_-$ where $R$ is reductive and $N$ is the unipotent part. Note that in the case $G=Sl(n+1)$ and $P$ as above, $N$ is normal in $P$ and $P$ is the semidirect product $P=N\rtimes R$. Moreover $R$ is an extension $$1\to Sl(n)\to R\to \C^\ast\to 1.$$ We denote $\mathcal{G}=\mathcal{H}\oplus\bigoplus \mathcal{G}_\alpha$ the Cartan decomposition of the Lie algebra $\mathcal{G}$ of $G$, with $\mathcal{H}$ the Cartan subalgebra of diagonal matrices and $\alpha$ varying in the set of roots of the form $\alpha_{ij}$, for $i,j=1,\ldots n+1$ and $i\not=j$. We recall that the corresponding root spaces $\mathcal{G}_{\alpha_{ij}}$ are generated by the matrices $E_{ij}$ with $1$ in the $(i,j)$ entry and $0$ elsewhere. We denote $\alpha_1=\alpha_{1,2},\ldots,\alpha_{n}=\alpha_{n,n+1}$ the usual choice of positive simple roots for $\mathcal{G}$. Following the notations of \cite{fuha}, we denote the linear forms on $\mathcal{H}$ corresponding to the diagonal entries by  $L_1,\ldots,L_{n+1}$. They satisfy the relation $L_1+\cdots+L_{n+1}=0$. We also denote $\lambda_1=L_1$, $ \lambda_2=L_1+L_2,\, \ldots,\, \lambda_n=L_1+\cdots+L_n$ the fundamental weights. Since $H$ is also a Cartan subalgebra for $\mathcal{P}$, one can  associate to $\mathcal{P}$ the same weight space as to $\mathcal{G}$, namely the sublattice of $\mathcal{H}^\ast$ generated by $\lambda_1,\ldots,\lambda_n$, cfr. \cite{ot}.

 Note the decomposition $\mathcal{N}_-=\mathcal{G}_{\alpha_{2,1}}\oplus\ldots\oplus\mathcal{G}_{\alpha_{n+1,1}}$ and that $\mathcal{N}_-$ is the set of matrices 
$$\left(\begin{array}{ll}0& \underline{0}\\ \underline{v}& \underline{0}\end{array}\right),$$ with $\underline{v}\in\C^n$.
We consider also the decomposition $\mathcal{G}=\mathcal{P}\oplus \mathcal{N}_+$ with $\mathcal{N}_+=\mathcal{G}_{\alpha_{1,2}}\oplus\ldots\oplus\mathcal{G}_{\alpha_{1,n+1}}$, the set of matrices of the form $$\left(\begin{array}{ll}0& \underline{w}^t\\ \underline{0}& \underline{0}\end{array}\right)$$ with $\underline{w}\in\C^n$. Observing that the tangent space at $P\in G/P$ can be identified with $\mathcal{N}_+\cong \C^n$, one can check that the tangent bundle $T_{\PP^n}$, defined in general by the adjoint representation $$ad:\P\to Gl(\mathcal{G}/\mathcal{P}),$$ can be described equivalently as the representation $$ad(p)(w)=\lambda\underline{w}^t A^{-1}$$ for $p$ an element in $P$ as in (\ref{eq:formP}). This means that $T_{\PP^n}$ is equal to the bundle associated to the dual of the standard representation of $Sl(n)\subset R$ tensorized with the line bundle $\OO(1)$. 

Similarly, one can check that the adjoint representation $$ad:P\to Gl(\mathcal{N}_-)$$ can be described by 
\begin{equation}\label{eq:omegarepr}ad(p)(v^t)=\lambda^{-1} A\underline{v},\end{equation} so it is isomorphic to the dual of the representation obtained above and associated to $T_{\PP^n}$. In conclusion, the adjoint representation of $P$ in $\mathcal{N}_-$ produces the vector bundle $\Omega_{\PP^n}$ and moreover $\Omega_{\PP^n}$ is equal to the bundle associated to the standard representation of $Sl(n)\subset R$, tensorized with $\OO(-1)$. Note that the maximal weight of the $P$-representation associated to $\Omega$ is $L_2-\lambda_1=\lambda_2-2\lambda_1=-\alpha_1$.

Now we consider arbitrary representations of $P$. One knows that any representation of $P$ is completely reducible if and only if its restriction to $N$ is trivial by a result of Ise in \cite{ise}, see also the notes \cite{ot}. This happens, for example, for the representations associated to $T_{\PP^n}$ and $\Omega_{\PP^n}$ calculated above. Moreover, since $\mathcal{R}\cong \C\oplus sl(n)$, any representation of $P$ restricted to $R$ becomes completely reducible. Conversely, any representation $\sigma$ of $R$ extends to a representation of $\rho$ of $P$ by letting $N$ act trivially. If $\sigma$ is irreducible then by a result of Ramanan in \cite{ram}, the associated vector bundle $E_{\rho}$ is stable and it is called a {\em irreducible} homogeneous vector bundle. For example one can obtain in this way the stability of $\Omega_{\PP^n}$, as well as the stability of all the bundles $S_{\mu}\Omega_{\PP^n}(k)$ with $S_{\mu}$ the Schur functor defined by an arbitrary partition $\mu=(p_1\geq p_2\geq\cdots\geq p_n)$ and arbitrary $k\in \Z$, since we know that $\Omega_{\PP^n}$ is a shift of the bundle associated to the standard representation of $Sl(n)\subset R$. Note that $\mathcal{H}$ is a Cartan subalgebra also for $\mathcal{R}$, so $\Z(\lambda_1,\ldots,\lambda_n)$ is the weight space for $\mathcal{R}$. Any irreducible representation of $R$ is the tensor product of a irreducible representation of $Sl(n)\subset R$ and a irreducible representation of $\C^\ast$. The representations of $R$ induced from $Sl(n)$ can be described by a maximal weight vector $\lambda=a_2\lambda_2+\cdots+a_n\lambda_n$ with $a_i\geq 0$ for any $i=2,\ldots, n$. On the other hand the representations of $R$ produced by its quotient $\C^\ast$ are described by a maximal weight vector of the form $k\lambda_1$ with arbitrary $k\in\Z$. So any irreducible representation of $R$ is described by a maximal weight vector of the form $k\lambda_1+\sum_{i=2}^n a_i\lambda_i$ with $a_i\geq 0$ for $i=2,\ldots, n$, and the set of such vectors is the Weil chamber of $R$. 

\subsection*{Graded vector bundles.}
If $E$ is a homogeneous vector bundle, then, as it was reminded above, it arises from some representation $\rho:P\to Gl(W)$ and one may consider the restriction $\sigma$ of this representation to $R$. Then one knows that $\sigma=\oplus \sigma_i$ a direct sum of irreducible representations. Extending $\sigma$ restriction back to $P$ by letting $N$ act trivially, one gets another vector bundle $gr(E)=\oplus E_i$ which is a direct sum of stable vector bundles, by Ramanan's result cited above. This is the graded vector bundle associated to $E$. We summarize the description of the associated graded bundle, as resulting from all the preceding discussion, in the following proposition.

\begin{prop}\label{prop:graded} Let $E$ is a homogeneous vector bundle and $W$ its associated $P$-module. Then the associated graded bundle $gr(E)$ is a direct sum of stable vector bundles $E_i$ of the form $$E_i=S_{\mu_i}\Omega\otimes\OO(d_i),$$ with $\mu_i=(m_2^{(i)}\geq\cdots\geq m_n^{(i)})$ a partition and $d_i$ an integer. This decomposition is obtained by considering $W$ as a $R$-module and decomposing it into a direct sum of irreducible $R$-modules $W_i$ of weights $(d_i-\sum_{i=2}^n m_i)L_1+m_2^{(i)}L_2+\cdots m_n^{(i)}L_n$. \end{prop}
As an application of the theory just recalled, and for later use in this article, we compute the graded bundles associated to any Schur functor $S_\mu (V\otimes\OO)$. Let us write $V=\C e_1\oplus U$, with $U\cong \C^n$ the $P$-invariant subspace generated by $e_2,\ldots,e_{n+1}$ and whose associated homogeneous bundle we know to be $\Omega(1)$. One has the following result, certainly well known to the experts, but for which we provide a proof for convenience of the reader.
\begin{prop}\label{prop:gradedschur} The graded homogeneous vector bundle associated to $S_\mu V\otimes \OO$ is the following
\begin{equation}\label{eq:gradschur}gr(S_\mu V\otimes \OO)=\bigoplus_\nu S_\nu [\Omega(1)]\otimes\OO(|\mu|-|\nu|)=\bigoplus_\nu S_\nu \Omega\otimes\OO(|\mu|),\end{equation}
the sum extended over the set of all $\nu\subset\mu$ with $\mu\setminus\nu$ having no two boxes in the same column (cfr. \cite{fuha}, exercise 6.12). \end{prop}
\begin{proof} The first case is $V\otimes\OO$ itself, whose associated representation $\rho: P\to Gl(V)$ is the restriction of the standard one. Recalling the formula (\ref{eq:omegarepr}) for the $R$-representation associated to $\Omega$ and restricting $\rho$ to $R$, one finds the decomposition
\begin{equation}\label{eq:grstandard}gr(V\otimes\OO)=\Omega(1)\oplus\OO(1).\end{equation}  Note that
one could have used as well the well known Euler sequence \begin{equation}\label{eq:euler}0\to\Omega(1)\to V\otimes\OO\to\OO(1)\to 0,\end{equation} observing that it is $G$-invariant and gives a filtration of $V\otimes\OO$ with irreducible quotients. To find the stratification of $S_\mu V\otimes\OO$ we proceed by induction on the lenght $|\mu|$. Formula (\ref{eq:gradschur}) has already been proved in the base case $|\mu|=1$. Assuming $|\mu|>1$, we compute the restriction of the $G$-representation $S_\mu V$ to $Sl(n)\subset P$ using the formula in \cite{fuha}, Exercise 6.12, from which we obtain  \begin{equation}\label{eq:resschur}\mbox{Res}_{Sl(n)}S_\mu V=\bigoplus S_{\nu}U\end{equation} with the sum extended to all those partitions $\nu\subset\mu$ such that $\mu\setminus\nu$ has no two boxes in the same column. This decomposition already says that $gr(S_\mu V\otimes\OO)$ is the direct sums of irreducible bundles of the form $S_\nu \Omega(d_\nu)$ with $\nu$ varying as above. To get the correct shifts $d_\nu$ we find it useful to consider the following construction. For any $\mu'$  with $|\mu\setminus\mu'|=1$
consider the Pieri inclusion $S_\mu V\to S_{\mu'}V\otimes V$, also called a Olver map, see \cite{samwey} for an explicit formula, and the $G$-invariant composition of sheaf maps \begin{equation}\label{eq:olvermap} S_\mu V\otimes\OO\to S_{\mu'} V\otimes V\otimes\OO\to S_{\mu'}V\otimes\OO(1),\end{equation} with the last map equal to the evaluation map. We consider also the direct product of all these sheaf morphisms $$S_\mu V\otimes\OO\stackrel{f}\to \bigoplus_{|\mu\setminus\mu'|=1}S_{\mu'}V\otimes\OO(1).$$  From the explicit formula for the Olver map and the decomposition (\ref{eq:resschur}) applied to $S_\mu V$ and to all the $S_{\mu'}V$ above, it follows that the kernel of $\theta$ is induced from the representation $S_\mu U$ of $Sl(n)\subset P$. So one gets the exact sequence
\begin{equation}\label{eq:beginresomega} 0\to S_\mu[\Omega(1)]\to S_\mu V\otimes\OO\stackrel{f}\to\bigoplus_{|\mu\setminus\mu'|=1}S_{\mu'}V\otimes\OO(1).\end{equation}
By the inductive hypothesis we know that $gr(S_{\mu'}V\otimes\OO(1))=\bigoplus S_{\nu'} \Omega\otimes \OO(|\mu'|+1)=\bigoplus S_{\nu'} \Omega\otimes \OO(|\mu|)$, with $\mu'\setminus\nu'$ having no two boxes in the same column. Taking into account the decomposition (\ref{eq:resschur}) that we already know, we get that $gr(S_\mu V\otimes\OO)$ is the direct sum $S_\mu[\Omega(1)]\oplus\bigoplus S_{\nu'} \Omega\otimes \OO(|\mu|)$, this last sum extended over all $\nu'\subset\mu'$ as above, with varying $\mu'$. This implies formula (\ref{eq:gradschur}) for the given $\mu$. \end{proof}
\subsection*{Quiver representations.}
We will now briefly recall the theory which establishes an equivalence between the category of homogeneous vector bundles on $\PP(V)$ and invariant morphisms, with some category of representations with relations of a given quiver, associated to $\PP(V)$. In the literature this theory is actually formulated for homogeneous vector bundles on various classes of rational homogeneous variety $X=G/P$, starting with the article \cite{boka}, but we limit our exposition only to the case $X=\PP(V)$. We start mentioning the following fact, a proof of which can be found in \cite{otru} Theorem 3.1.
Regarding $\mathcal{N}_-$ as a $\mathcal{R}$-module by the adjoint action, it is known that the $\mathcal{P}$-module structure on $W$ associated to the originary bundle $E$ is equivalent to the datum of the $\mathcal{R}$-module structure on $W$, associated to $gr(E)$, plus a $\mathcal{R}$-equivariant map $\theta:\mathcal{N}_-\otimes W\to W$ satisfying the property $0=\theta\wedge\theta:\bigwedge^2 \mathcal{N}_-\otimes W\to W$.  This is the same as a $G$-equivariant morphism $\theta: \Omega\otimes gr(E)\to gr(E)$ satisfying $0=\theta\wedge\theta:\bigwedge^2\Omega\otimes gr(E)\to gr(E)$. The morphism $\theta$ is zero if and only if $E=gr(E)$, that is $E$ is a completely reducible vector bundle. 

\vskip2mm
Now let us fix a maximal weight vector $\eta_i$ for the $R$-representation associated to any irreducible bundle $E_i$ and let us denote $\xi_1,\ldots,\xi_n$ the weights of $\Omega$.
Writing $gr(E)=\bigoplus V_i\otimes E_i$, with $E_i$ the distinct irreducible bundles in the decomposition of $gr(E)$ and $V_i$ vector spaces accounting for the multiplicity of the appearance of $E_i$ in $gr(E)$, one decomposes a $R$-morphism $\theta$ as above into $R$-invariant morphisms $\theta_{ij}:V_i\otimes \Omega\otimes E_i\to V_j\otimes E_j$. A key observation is the following (cfr. \cite{otru} Theorem 4.3, Definition 5.5 and Corollary 5.4):
\begin{equation}\label{ex:ext1irred}\dim Hom^G(\Omega\otimes E_i,E_j)=\left\{\begin{array}{lll}1 & \mbox{ if there exists $k$ s.t. }\eta_j=\eta_i+\xi_k \\0& \mbox{ otherwise }\end{array}\right.\end{equation}
Moreover one has the isomorphism
$$Hom^G(\Omega\otimes E_i,E_j)\cong Ext^1(E_i,E_j)^G,$$ see \cite{otru}, Theorem 4.3.
In our context of $Sl(n)$ representations, we know that $E_i$ and $E_j$ have the form $S_\mu\Omega(h)$ and $S_{\mu'}\Omega(k)$, respectively, and one can rephrase the observation above saying that $Hom^G(\Omega\otimes E_i,E_j)=\C$ if and only if $$Hom^G(\Omega(1)\otimes S_\mu[\Omega(1)](h-|\mu|),S_{\mu'}[\Omega(1)](k-|\mu'|+1))=\C.$$ Passing to the associated $R$-representations, and denoting with $L$ the $R$-representation associated to $\OO(1)$, we get that the above equality holds if and only if $$Hom^R(U\otimes S_\mu U\otimes L^{\otimes (h-|\mu|)}, S_{\mu'} U\otimes L^{\otimes(k-|\mu'|+1)})=\C.$$ Now one can apply Pieri formula and conclude that the above holds if and only if $|\mu'\setminus\mu|=1$ and $h=k$ and $Hom^G(\Omega\otimes E_i,E_j)=0$ otherwise. In conclusion we have 
\begin{equation}\label{eq:arrows} Hom^G(\Omega\otimes E_i,E_j)\not=0 \iff\left\{
\begin{array}{ll}E_i=S_\mu\Omega(h)\\ 
E_j=S_{\mu'}\Omega(h)\\ \mu\subset\mu' \mbox{ and }|\mu'\setminus\mu|=1\end{array}\right. \end{equation}  One then fixes a $G$-invariant morphism $m_{ij}\in Hom^G(\Omega\otimes E_i,E_j)\cong\C$ in such way that $m_{ij}(\xi_k\otimes\eta_i)=\eta_j$ whenever $\eta_j=\eta_i+\xi_k$.

\noindent
Going back to the  case of a general $P$-structure on $gr(E)$, we get that any $R$-invariant morphism $\theta\in Hom^G(\Omega\otimes gr(E),gr(E))$ can be written as $\theta=\bigoplus \theta_{ij}=\bigoplus (f_{ij} \otimes m_{ij})$, with $f_{ij}:V_i\to V_j$ linear maps of vector spaces. 
\subsection*{The quiver $\mathcal{Q}_{\PP^n}$.}
One defines the quiver $\mathcal{Q}_{\PP^n}$ associated to $\PP^n$ as the oriented graph with vertices $[E_i]$, indexed by the set of irreducible bundles, and directed edges $[E_i]\to [E_j]$ iff $Hom^G(\Omega\otimes E_i,E_j)=\C$. In a similar way one defines the quivers $\mathcal{Q}_X$ associated to more general rational homogeneous varieties, see \cite{boka}, \cite{otru}.
\begin{obs}\label{obs:quiverord} {\em In \cite{otru} Corollary 5.3 it is shown that for any arrow $[E_i]\to [E_j]$ in $\mathcal{Q}_{\PP^n}$ one has $\mu(E_j)=\mu(E_i)+\mu(\Omega)$, and in particular the quiver $\mathcal{Q}_{\PP^n}$ is {\em levelled}, see Definition 5.1 in \cite{otru}. Moreover we can consider $\mathcal{Q}_{\PP^n}$ as a poset with ordering defined by $[E_i]<[E_j]$ iff $E_i=S_\nu\Omega(h)$ and $E_j=S_\mu\Omega(h)$ with $\nu\subset\mu$, equivalently if there exists a sequence of arrows in $\mathcal{Q}_{\PP^n}$ starting at $[E_i]$ and ending at $[E_j]$. Note that in this ordering any $\nu\not=\eta$ with $|\nu|=|\eta|$ produce incomparable elements. }\end{obs}
\subsection*{The quiver representation of a homogeneous bundle.}
Given the homogeneous bundle $E$, one constructs a linear representation $[E]$ of the quiver $\mathcal{Q}_{\PP^n}$ in the following way. To any vertex $[E_i]\in \mathcal{Q}_{\PP^n}$ one associates the vector space $V_i$ defined above, and to any arrow $[E_i]\to [E_j]$ the linear map $f_{ij}:V_i\to V_j$ defined above. The condition $\theta\wedge\theta=0$ imposes some relations between all these maps, i.e. some relations between compositions of the $f_{ij}$'s. We do not enter into the details of the construction of relations, we just mention that in \cite{otru} it was proved that for projective spaces,  and more generally for grassmannians, they are generated by the {\em commutativity} relations associated to any square diagram of $f_{ij}$'s.
A morphism between two representations $(\{V_i\},\{f_{ij}\})$ and $(\{W_i\},\{g_{ij}\})$ of  a quiver is in the natural way a set of maps $\psi_i:V_i\to W_i$ such that $\psi_j\circ f_{ij}=g_{ij}\circ \psi_i$ for any $i,j$.
The general result of this theory, stated for projective spaces, is the following (cfr. \cite{otru}, Theorem 5.9)
\begin{thm}\label{thm:quiverequivalence} There is an equivalence of categories between the category of homogeneous vector bundles on $\PP^n$ and invariant morphisms, and the category of representations  with the given relations of the quiver $\mathcal{Q}_{\PP^n}$.\end{thm}
\begin{definition}\label{def:support} Given a homogeneous vector bundle $E$, the {\em support} of $E$ is the subset of vertices of $\mathcal{Q}_{\PP^n}$ defined by supp$(E)=\{[E_i]\ |\ V_i\not=0\}$. We also denote  with $\mathcal{Q}_{|E}$ the full subquiver of $\mathcal{Q}_{\PP^n}$ with set of vertices supp$(E)$ and $\mathcal{Q}(E)$ the subquiver of $\mathcal{Q}$, with the same vertices as $\mathcal{Q}_{|E}$ and arrows $[E_i]\to [E_j]$ iff $f_{ij}\not=0$ . \end{definition}
It will be useful to explain the effect of a non-zero map $f_{ij}$ in the representation $[E]$ associated to a homogeneous bundle $E$. We do this in the simple case $gr(E)=E_i\oplus E_j$ and $\theta:\Omega\otimes gr(E)\to gr(E)$ corresponding to the choice of a non-zero element in $Hom^G(\Omega\otimes E_i, E_j)$ as discussed above. Then the representation associated to $E$ has $\mathcal{Q}(E)$ equal to the subgraph $[E_i]\to [E_j]$ of $\mathcal{Q}_{\PP^n}$ and it has $[E_j]$ as sub-representation and $[E_i]$ as quotient. Consequently, $E$ is in a $G$-invariant extension $$0\to E_j\to E\to E_i\to 0,$$ which is not split since $E$ is not completely reducible, because $\theta\not=0$. So the given choice of a non-zero element in $Hom^G(\Omega\otimes E_i, E_j)$ is equivalent to the choice of a non-zero element in $Ext^1(E_i,E_j)^G$, which is then isomorphic to $Hom^G(\Omega\otimes E_i, E_j)\cong\C$.
\subsection*{Slopes of quiver representations and stability}
Given any reducible bundle $gr(E)=\bigoplus V_\lambda\otimes E_\lambda$,  let us denote $\alpha_\lambda=\dim V_\lambda$ and $\alpha=(\alpha_\lambda)$ the dimension vector. Let us denote $K_0(R-mod)$ the Grothendieck group of $R$-representations. This is the free $\Z$-module generated by the isomorphism classes of the irreducible bundles $E_\lambda$, the isomorphism being defined by $[gr(E)]\mapsto \sum\alpha_\lambda E_\lambda$. Fixing a graded $F=gr(E)$ and following \cite{king} and \cite{otru} section 7, we introduce the natural homomorphism $\mu_F:K_0(R-mod)\to \Z$, associated to $F$ and defined by \begin{equation}\label{eq:muF}\mu_F(E_\lambda)=c_1(F)rk(E_\lambda)-rk(F)c_1(E_\lambda).\end{equation} Note that $\mu_F(E)=0$.
We introduce the following definition, taken from  \cite{king}, Definition 1.1.
\begin{definition}\label{def:stability} The quiver representation $[E]$ associated to $E$ is said to be $\mu$-semistable if for any subrepresentation $[E']$ one has  $\mu_F(gr(E'))\geq 0$. It is called $\mu$-stable if, in addition,  the only subrepresentations $[E']$ for which $\mu_F(gr(E'))=0$ are $[E]$ and $[0]$. \end{definition}
The result that connects $\mu$-(semi)stability of quiver representations and (semi)stability of homogeneous bundle is the following, whose formulation we extract from \cite{otru}, theorems 7.1 and 7.2 and the references therein.
\begin{thm}[cfr. \cite{otru}, theorems 7.1 and 7.2]\label{thm:rohfai} A homogeneous vector bundle $E$ is semistable if and only if its associated quiver representation $[E]$ is $\mu$-semistable. Moreover $[E]$ is $\mu$-stable if and only if $E=W\otimes F$ with $F$ stable homogeneous and $W$ a irreducible $G$-vector space.\end{thm}
In the next result we complete the study of the $P$-module structure of the homogeneous bundles $S_\mu V\otimes \OO$ by describing their associated quiver representation.
\begin{prop}\label{prop:schurquiver} Let $\mathcal{E}=S_\mu V\otimes \OO$ with $\mu$ any partition. Then $supp([\mathcal{E}])$ is the set of all $[E_{\nu}]=[S_\nu\Omega(|\mu|)]$ such that $\nu\subset\mu$ and   $\mu\setminus\nu$ has no two boxes in the same column. We have $\dim V_\nu=1$ for any such $\nu$. Moreover $f_{\nu,\eta}\not=0$ for any $\nu,\eta$ as above and $|\eta\setminus \nu|=1$. In particular $\mathcal{Q}(\mathcal{E})=\mathcal{Q}_{|\mathcal{E}}$ and it is a connected graph.
\end{prop}

For the proof we need the following result.
\begin{lm}\label{lemma:quivertensor} Let $E$ and $F$ be two homogeneous bundles with $gr(E)=\bigoplus V_i\otimes E_i$ and $gr(F)=\bigoplus W_j\otimes F_j$ and quiver representations with maps $e_{ij}:V_i\to V_j$ and $f_{ij}:W_i\to W_j$, respectively. Let us also write $gr(E\otimes F)=\bigoplus V_G\otimes G$. Assume $e_{ij}\not=0$ for some $i,j$ and consider any $R$-irreducible summand $G\subset E_j\otimes F_k$. Then there exists some irreducible summand  $G'\subset E_i\otimes F_k$ such that in the quiver representation associated to $E\otimes F$ the linear map $f_{G'G}:V_{G'}\to V_{G}$ is non-zero. \end{lm}
\begin{proof} Consider $e_i$, $e_j$ some maximal weight vectors with associated weights $\eta_i$, $\eta_j$, associated to $E_i$, $E_j$ respectively and $n$ a weight vector in $\mathcal{N}_{-}$ with associated weight $\xi_h$, such that $\eta_j=\eta_i+\xi_h$. One can also assume that by means the Lie algebra action of $\mathcal{N}_{-}$ on $gr(E)$ one gets $0\not=n\cdot (v\otimes e_i)=e_{ij}(v)\otimes e_j\in V_j\otimes E_j$. Recall that $gr(E\otimes F)=\bigoplus  (M_{ikl}\otimes V_i\otimes W_k)\otimes E_{ikl}$, given a decomposition $E_i\otimes F_k=\bigoplus M_{ikl}\otimes E_{ikl}$ by the Littlewood-Richardson formula. Then the $\mathcal{N}_{-}$-action on $gr(E\otimes F)$ is given by $$n\cdot (v\otimes e_i \otimes w\otimes f_k)=\theta_E(v\otimes e_i)\otimes w\otimes f_k+v\otimes e_i\otimes\theta_F(w\otimes f_k),$$ with $\theta_E:\mathcal{N}_{-}\otimes gr(E)\to gr(E)$ and $\theta_F:\mathcal{N}_{-}\otimes gr(F)\to gr(F)$ the structure maps essociated to $E$ and $F$. 

Now assume that a maximal weight vector for $G$ is given as $e_j\otimes f_k$, with $f_k\in F_k$ some weight vector with associated weight $\psi_k$. The Lie algebra action on the tensor product gives $n\cdot(v\otimes e_i\otimes w\otimes f_k)=e_{ij}(v)\otimes e_j\otimes w\otimes f_k+v\otimes e_i\otimes n\cdot(w\otimes f_k)$.  Note that the summands $e_{ij}(v)\otimes e_j\otimes w\otimes f_k$ and  $e_i\otimes n\cdot (w\otimes f_k)$ belong to linearly disjoint $R$-invariant subspaces of $E\otimes F$ and that $e_j\otimes f_k \not=0$. Note that $e_{ij}(v)\otimes e_j\otimes w\otimes f_k$ belongs to the isotypic component $ M_{ij G'}\otimes V_j\otimes W_k \otimes G\subset E_j\otimes F_k$ and that we can assume $e_{ij}(v)\not=0$ by hypothesis. Moreover $v\otimes e_i\otimes w \otimes f_k$ has a non-zero component in some isotypic component $M_{G'G}\otimes V_i\otimes W_k\otimes G'$ of $V_i\otimes W_k\otimes E_j\otimes F_k$, which completes the proof. \end{proof}

\begin{proof}[Proof of Proposition \ref{prop:schurquiver}] The statements about the support of $[\mathcal{E}]$ and the multiplicity spaces $V_\nu$ have already been proved in Proposition \ref{prop:gradedschur}. We use induction on $|\mu|$ to prove the statements about the arrows. The case $|\mu|=1$ follows from the existence of the Euler sequence (\ref{eq:euler}). Note that $\mathcal{Q}(\mathcal{E})$ is the graph $[\OO(1)]\to[\Omega(1)]$ in this case. 
Now assume the statement true for any partition $\nu$ with $1\leq |\nu|<|\mu|$. Given $\mu$ take a partition $\nu\subset\mu$ such that $|\mu\setminus\nu|=1$. Then there exists a $G$-invariant projection, given by the Pieri decomposition,
$$S_\nu V\otimes V\otimes\OO\to S_\mu V\otimes \OO.$$
Recall that $gr(V\otimes\OO)=\Omega(1)\oplus\OO(1)$ and $gr(S_\nu V\otimes\OO)=\bigoplus S_{\nu'}\Omega(|\nu|)$ with $\nu'$ varying as described in Proposition \ref{prop:gradedschur}. Using the Pieri decomposition relative to $R$, one can see that $$gr(S_\nu V\otimes V\otimes\OO)=\bigoplus \C^2\otimes S_{\nu'}\Omega(|\nu|+1)\oplus\bigoplus_{\eta\supset\nu,\ |\eta\setminus\nu|=1}S_{\eta}\Omega(|\nu|+1).$$ 
Note that the first summand of the formula above is equal to $gr(S_\nu V\otimes\OO(1))$ taken with multiplicity $2$, and that the second summand contains the term $S_{\mu}\Omega(|\mu|)$, appearing also in $gr(\mathcal{E})$. We apply Lemma \ref{lemma:quivertensor} setting $F=S_\nu V\otimes\OO$, $E=V\otimes\OO$, $E_i=\OO(1)$, $E_j=\Omega(1)$, $F_k=S_{\nu}\Omega(|\nu|)$ and $G=S_{\mu}\Omega(|\mu|)\subset E_j\otimes F_k$. The outcome is a non-zero arrow $[G']\to [G]$ in the quiver $\mathcal{Q}(E\otimes F)$, with $G'$ necessarily equal to $S_{\nu}\Omega(|\nu|+1)=S_{\nu}\Omega(|\mu|)$. 

\noindent
Secondly, note that the quiver $\mathcal{Q}(S_{\nu} V\otimes\OO(1))$ is contained in $\mathcal{Q}(S_{\nu} V\otimes V\otimes\OO)$, since the first bundle is a quotient of the latter. Hence by inductive hypothesis any arrow $[E_i(1)]\to[E_j(1)]$ is present in $\mathcal{Q}(S_{\nu} V\otimes V\otimes\OO)$. Now, since $\mathcal{E}$ is a quotient of $S_{\nu} V\otimes V\otimes\OO$, then the same holds for their associated representations. In particular it follows that any non-zero arrow $[E_i]\to[E_j]$ in $\mathcal{Q}(S_{\nu} V\otimes V\otimes\OO)$ is sent to a non-zero arrow in $\mathcal{Q}(\mathcal{E})$, provided that $[E_i]$ and $[E_j]$ belong to $Supp(\mathcal{E})$. Putting together the result obtained above about the arrows in $\mathcal{Q}(S_{\nu} V\otimes V\otimes\OO)$, we see that in  $\mathcal{Q}(\mathcal{E})$ one finds the arrow $[S_{\nu}\Omega(|\mu|)]\to [S_{\mu}\Omega(|\mu|)]$, and any arrow $[E_i]\to[E_j]$ already appearing in $\mathcal{Q}(S_{\nu} V\otimes\OO(1))$ with $[E_i],[E_j]\in Supp(\mathcal{E})$. By letting $\nu$ vary arbitrarily under the conditions $\nu\subset\mu$ and $|\mu\setminus\nu|=1$, and observing that $Supp(\mathcal{E})\setminus \{[S_{\nu}\Omega(|\mu|)]\}$ is covered by the sets $Supp(S_{\nu} V\otimes\OO(1))$, the proposition follows.
\end{proof}
\section{The structure of $\P^k(\OO_{\PP^N}(d))$}\label{sec:thm1}
We assume $n\geq 1$ and $d\in\Z$. As in the preceding sections, we will denote $\OO_{\PP^n}$ with $\OO$. The following result is well known.
\begin{prop}[see \cite{perk} page 37, \cite{som} Proposition 2.2]\label{prop:dgeqk} If either $d\geq k$ or $d<0$ one has $\P^k\OO(d)\cong S^kV\otimes\OO(d-k)$. \end{prop}
The approach of \cite{perk} for proving the result above is based on the existence of the differential operators
$$ \xi^k: \C(x_0,\ldots,x_n)\to S^k V\otimes \C(x_0,\ldots,x_n) $$ defined by
\begin{equation}\label{eq:diffk} \xi^k(f)=\sum_{|I|=k}{k\choose I}x^I\otimes \partial^I(f),\end{equation} with $I=(i_0,\ldots,i_n)$ denoting multi-indexes of degree $|I|=i_0+\cdots+i_n=k$ and ${k\choose I}=k!/(i_0!\cdots i_n!)$. Note that the operator above is the $k$-th power of the Euler operator $\xi=\sum x_i\otimes \partial_i$, with $\partial_i=\partial_{x_i}$, as an element of the commutative ring $\C[x_0,\ldots,x_n]\otimes\C[\partial_0,\ldots,\partial_n]$. We are specially interested in the restriction of $\xi^k$ to the line bundle $\OO(d)$, which one can view as a subsheaf of the constant sheaf with coefficients in $\C(x_0,\ldots,x_n)$, identifying the local sections $f\in \OO(d)$ with those fractions $f\in\C(x_0,\ldots,x_n)$ which are homogeneous with homogeneity degree $d$. In this way one gets an operator $$\xi^k:\OO(d)\to S^k V\otimes \OO(d-k).$$  
Note that $\xi^k$ is $Sl(n+1)$-invariant for any $k\geq 1$. By the universal property of principal parts bundles with respect to differential operators, one gets the $Sl(n+1)$-invariant sheaf morphisms $$u_k:\P^k\OO(d)\to S^kV\otimes\OO(d-k),$$ which are shown to be isomorphisms in the cases stated in Proposition \ref{prop:dgeqk}. On the other hand, if $0\leq d<k$ it is shown in \cite{perk} p. 37 that $u_k$ is not an isomorphism, and more precisely one can prove the following results.
\begin{thm}\label{thm:splitting} Assume $0\leq d<k$, then the following facts hold.

\noindent
{\em 1.} The canonical sequence $$0\to Q_{k,d}\to \P^k\OO(d)\to \P^d\OO(d)\to 0$$ is a split exact sequence, inducing an isomorphism
$$\P^k\OO(d)\cong Q_{k,d}\oplus \P^d\OO(d)=Q_{k,d}\oplus (S^dV\otimes \OO).$$

\noindent
{\em 2.} The morphism $u_k:\P^k\OO(d)\to S^kV\otimes \OO(d-k)$ induces an isomorphism $Q_{k,d}\cong\ker(\eta^{k-d} :S^kV\otimes \OO(d-k)\to S^dV\otimes \OO),$ with $\eta=\sum\partial_i\otimes x_i$ and the power $\eta^k$ defined in the ring $\C[\partial_{x_0},\ldots,\partial_{x_n}]\otimes\C[x_0,\ldots,x_n]$.
\end{thm}
A proof of this result has been already been sketched in \cite{perk}, as mentioned above, but we think it useful to provide a new one, substantially based on the same ideas, but using more explicit formulas, one of which we state in the following lemma.
\begin{lm}\label{lemma:compoperators} Under the notations introduced above, the following formula holds, for any $k> 0$ and $f\in\C(x_0,\ldots,x_n)$ homogeneous:
\begin{equation}\label{eq:compoper} \eta(\xi^k(f))=k(\deg(f)-k+1)\xi^{k-1}(f).\end{equation}\end{lm}
\begin{proof} We calculate \begin{eqnarray*}
\eta(\xi^k(f))&=&\left(\sum \partial_j\otimes x_j\right)\left(\sum_I {k\choose I}x^I\otimes\partial_I f\right)\\
&=&\sum_{j,I}{k\choose I}\partial_jx^I\otimes x_j\partial_I(f)\\
&=&\sum_{j,I}{k\choose I} i_j (x^I/x_j)\otimes x_j\partial_I f\\
&=&\sum_{j,I\ :\ I-e_j\geq \underline{0}} k {k-1\choose I-e_j} x^{I-e_j}\otimes x_j\partial_I f\\
&&\mbox{with $e_j$ the $j$-th vector of the canonical basis of $\Z^{n+1}$}\\
&=&k\sum_{j,J}{k-1\choose J} x^J\otimes x_j\partial_{J+e_j} f\\
&&\mbox{with the sum over those $J$ s.t.$|J|=k-1$}\\
&=&k\sum_{J}{k-1\choose J} \deg(\partial_{J} f)x^J\otimes\partial_{J} f\\
&=&k(\deg(f)-k+1)\xi^{k-1}(f).\end{eqnarray*}
\end{proof}
\begin{proof}[Proof of Theorem \ref{thm:splitting}]
 Using Proposition \ref{prop:dgeqk}, one knows that $\P^d\OO(d)\cong S^dV\otimes \OO$, so the canonical sequence in the statement becomes $$0\to Q_{k,d}\to \P^k\OO(d)\to S^dV\otimes \OO\to 0.$$ Recall also that the composition $$H^0\P^k\OO(d)\to H^0\P^d\OO(d)\stackrel{\cong}\to H^0\OO(d)$$ is surjective, so the exact sequence above is exact at the level of global sections. Finally, since $P^d\OO(d)$ is trivial, one obtains the claimed splitting {\em 1}.
\vskip2mm
\noindent
{\em 2.}  We need to carefully analyze the morphism $u_k: \P^k\OO(d)\to S^kV\otimes\OO(d-k)$. By the homogeneity of $\P^k\OO(d)$ and the invariance of $u_k$ one can concentrate oneself in calculating the effect of $u_k$ at the fibre of $\P^k\OO(d)$ at $p=(0:\cdots:0:1)\in D(x_n)$. We denote $z_i=x_i/x_n$ for $i=0,\ldots,n-1$ the standard set of parameters at $p$ and $x_n^d$ the basis of $\OO(d)$ over any point of $D(x_n)$. An arbitrary section of $\OO(d)$ regular over $D(x_n)$ has the form $P/x_n^h$ with $P=x_n^{d+h}P_0+\cdots P_{d+h}$ a homogeneous polynomial in $x_0,\ldots,x_n$ of degree $d+h$, with $h\geq 0$. A generating set for the fiber of $\P^k\OO(d)$ at $p$ is therefore given by $d_k(P/x_n^h)|_p=\sum_j P_j(z_0,\ldots,z_{n-1})x_n^{d}$, obtained by applying formula (\ref{eq:dk}) and setting $x_n^d$ the basis of $\OO(d)$ on $D(x_n)$. Note in particular that a basis for the fiber of $\P^k\OO(d)$ at $p$ is given by $$\{z^Jx_n^d\ :\ |J|\leq k\}.$$ We have $u_k(d_k(P/x_n^h))=\xi^k(P/x_n^h)=\sum_{j=0}^k{k\choose j}\xi^j(1/x^h)\xi^{k-j}(P)$, by applying the iterated Leibnitz formula for the power of the derivation $\xi$ with values in the ring $\C(x_0,\ldots,x_n)^{\otimes 2}$. 

If $h=0$ one gets $\xi^k(P)=0$, since $d<k$ and all $k$-th order derivatives appearing in $\xi^k(P)$ vanish. This means that the summand $S^dV\otimes\OO$ of $\P^k\OO(d)$ is contained in the kernel of $u_k$. Hence Im$u_k=u_k(Q_{k,d})$. 

On the other hand, take a particular $P/x_n^h$ with $h>0$ of the form $P/x_n^h=x_n^{d-t}x^J$, with $J=(j_0,\ldots,j_{n-1})$ and $d<|J|=t\leq k$. Then \begin{eqnarray*} \xi^k(P/x_n^h)&=&\sum_{j=0}^k{k\choose j}\xi^j(x_n^{d-t})\xi^{k-j}(x^J)\\
&=& \sum_{j=0}^k\left({k\choose j} \prod_{l=0}^{j-1}(d-t-l)\right)(x_n^j\otimes x_n^{d-t-j})\xi^{k-j}(x^J).\end{eqnarray*}
Note that the terms containing $\xi^{k-j}(x^J)$ vanish for $k-j>t$.
Using the monomial ordering deglex with $x_0>\cdots>x_n$ on the first factor of $\C[x_0,\ldots,x_n]\otimes\C[x_0,\ldots,x_n]$, we see that 
$\xi^k(P/x_n^h)$ has a maximum term of the form $c x^Jx_n^{k-t}\otimes x_n^{d-k}$, with $c$ a non-zero constant. By letting $x^J$ vary among all the monomials in $x_0,\ldots,x_{n-1}$ of degree $t\leq k$ we see that $\xi^k(x_n^{d-t}x^J)$ form an independent set in the fiber of $S^kV\otimes \OO(d-k)$ at $p$. Observing that $d_k(x_n^{d-t}x^J)=z^J x_n^d$, we get that $u_k(z^J x_n^d)$ with $d<|J|\leq k$ are linearly independent.

\vskip1mm
Now we claim that $\eta^{k-d}\circ u_k=0$. Indeed it is sufficient to show that $\eta^{k-d}\circ \xi^k=0$ and this follows from an iterated application of Lemma \ref{lemma:compoperators}, showing that $$\eta^{k-d}(\xi^k(f))=\left(\prod_{i=0}^{k-d-1}(k-i)(d-k+i+1)\right)\xi^{d}(f)=0.$$ Hence Im$(u_k)\subset\ker\eta^{k-d}.$ Observe that \begin{eqnarray*}
\dim\ker\eta^{k-d}&=&\mbox{rk}(S^kV\otimes\OO(d-k))-\mbox{rk}(S^dV\otimes \OO)\\
&=&\mbox{rk} (\P^k\OO(d))- \mbox{rk}(\P^d\OO(d))\\&=&\mbox{rk} (Q_{k,d})\\
&=&{k+n\choose n}-{d+n\choose n}\end{eqnarray*} and note moreover that 
$${k+n\choose n}-{d+n\choose n}=\sum_{i=0}^{k-d-1} {k-i+n-1\choose n-1}.$$ Precedingly we showed  that in the fiber of $u_k(Q_{k,d})$ one can find linearly independent elements of the form $u_k(z^J x_n^d)$ for any $d<|J|\leq k$, whose number is exactly the number above, hence $u_k$ maps the fiber of $Q_{k,d}$ isomorphically to the fiber of $\ker\eta^{k-d}$, which completes the proof.
 \end{proof}
 With the same techniques as above one obtains the following proof of Proposition \ref{prop:dgeqk}, with some extra precisions which will be useful later.
 \begin{prop}\label{prop:dgeqkbis} For any $k\leq d$ the map $$u_k: \P^k\OO(d)\to S^k V\otimes \OO(d-k)$$ is an isomorphism and there exists a non zero constant $c$ such that the following diagram is commutative
\[\begin{CD} \P^k\OO(d) @>>>  \P^{k-1}\OO(d) \\ @V{u_k}VV  @V{cu_{k-1}}VV\\ S^k V\otimes \OO(d-k) @>\eta>> S^{k-1} V\otimes \OO(d-k+1)\end{CD}.\]\end{prop}
\begin{proof}(sketch) Under the assumption $d\geq k$ and using similar calculations as in the preceding proof one can shows the injectivity of $u_k$ and hence that it is an isomorphism. The commutativity of the diagram above and the precise value of the constant $c$ are given by Lemma \ref{lemma:compoperators}. \end{proof}
We study the stability of the bundle $Q_{k,d}$. Preliminarly we observe that $Q_{k,d}$ is a homogeneous bundle with associated graded bundle and quiver representation described by the following proposition.
 \begin{prop}\label{prop:quiverpp} The bundle $Q_{k,d}$ has an invariant filtration $0=F_{k+1}\subset F_{k}\subset\cdots\subset F_{d+1}=Q_{k,d}$ with quotients $F_{i}/F_{i+1}=S^i\Omega(d)$ for $i=d+1,\ldots,k$. In particular $gr(Q_{k,d})=\bigoplus_{i=d+1}^k S^i\Omega(d)$.
 The quiver representation associated to $Q_{k,d}$ has multiplicity spaces $V_\lambda=\C$ and the quiver $\mathcal{Q}(Q_{k,d})$ is the following:
 $$  [S^{d+1}\Omega(d)]\to\cdots\to[S^{k}\Omega(d)].$$
  \end{prop}
  \begin{proof} The existence of the filtration above is obtained factorizing the canonical surjection $\P^k\OO(d)\to \P^d\OO(d)$ in the following composition $$\P^k\OO(d)\to \P^{k-1}\OO(d)\to\cdots\to \P^d\OO(d)$$ and recalling that the kernel of any map $\P^i\OO(d)\to \P^{i-1}\OO(d)$ is $S^i\Omega(d)$. 

To show the assertion about the quiver representation of $Q_{k,d}$ we use the identification $Q_{k,d}\cong \ker\eta^{k-d}$ given by Theorem \ref{thm:splitting}. From this it follows that the quiver associated to $Q_{k,d}$ is the {\em difference} of the quivers associated to $S^kV\otimes\OO(d-k)$ and to $S^dV\otimes\OO$, which we know given by $[S^{0}\Omega(d)]\to\cdots\to[S^{k}\Omega(d)]$ and $[S^{0}\Omega(d)]\to\cdots\to[S^{d}\Omega(d)]$, respectively,  as a consequence of Proposition \ref{prop:schurquiver}. The stated formula for  $\mathcal{Q}(Q_{k,d})$ follows.
 \end{proof}
 \begin{cor}\label{cor:ppquiver} The quiver representations associated to $\P^k\OO(d)$ for varying $k$ and $d$ have all multiplicity spaces $V_\lambda$
 with $\dim V_\lambda=1$ and the graphs $\mathcal{Q}(\P^k\OO(d))$ are the following:
\vskip2mm 
\mbox{$[S^{0}\Omega(d)]\to\cdots\to[S^{k}\Omega(d)]$}  \mbox{ for $d\geq k\geq 0$ or $d<0$};
\vskip2mm 
  $[S^{0}\Omega(d)]\to\cdots\to[S^{d}\Omega(d)]
 \cup [S^{d+1}\Omega(d)]\to\cdots\to[S^{k}\Omega(d)]$  
 
 for $k>d\geq 0$.
 \end{cor}
 \begin{proof} The cases $d\geq k\geq 0$ and $d<0$ are dealt with the help of Proposition \ref{prop:dgeqk}, which gives the isomorphism $\P^k\OO(d)\cong S^kV\otimes\OO(d-k)$ and with Proposition \ref{prop:schurquiver} applied to $S_\mu V=S^dV$.
 
 The statement in the case $k>d\geq 0$ is a consequence of the isomorphism $\P^k\OO(d)\cong Q_{k,d}\oplus (S^dV\otimes \OO)$, which implies that the associated quiver representation is the direct sum of the quiver representations associated to $Q_{k,d}$ and $S^dV\otimes \OO$. \end{proof}
 We are now in position to show the following stability result for $Q_{k,d}$, which, together with the results proved in Theorem \ref{thm:splitting} and Proposition \ref{prop:quiverpp}, completes the proof of Theorem \ref{thm:main1}.
 \begin{prop}[end of proof of Theorem \ref{thm:main1}] $Q_{k,d}$ is a stable homogeneous bundle for $k>d$ and $n\geq 2$.
 \end{prop}
 \begin{proof} In view of the result stated in Theorem \ref{thm:rohfai} it is sufficient to show that the quiver representation $[Q_{k,d}]$ is $\mu$-stable as in Definition \ref{def:stability} and that $Q_{k,d}$ is not isomorphic to any bundle of the form $S_\mu V\otimes E$ with $E$ stable homogeneous. 
  Writing $F=gr(Q_{k,d})$ and adopting the notations preceding Definition \ref{def:stability}, we have to show that $\mu_F(E')\geq 0$ for any $E'$ corresponding to a subrepresentation of $[Q_{k,d}]$, and moreover that $\mu_F(E')=0$ only for $E'=[Q_{k,d}]$ or $0$. Using the description given in Proposition \ref{prop:quiverpp}, one can see that the non zero subrepresentations of $[Q_{k,d}]$ have the form 
  \begin{equation}\label{eq:subQ}[S^{i}\Omega(d)]\to\cdots\to[S^{k}\Omega(d)],\end{equation} with $i\geq d+1$. Denoting with $\mu(\mathcal{E})=c_1(\mathcal{E})/\mbox{rk}(\mathcal{E})$ the ordinary slope of sheaves in $\PP^n$, we recall that $\mu(\Omega)=-1-1/n<0$ and that $\mu(S^i\Omega(d))=i\mu(\Omega)+d$,  from which we get the inequalities $$ \mu(S^{d+1}\Omega(d))>\cdots>\mu(S^k\Omega(d)).$$ Applying the standard inequalities on slopes, it follows that $\mu(F)\geq \mu(S^{i}\Omega(d)\oplus\cdots\oplus S^{k}\Omega(d))$ for any $d+1\leq i\leq k$, with equality only for $i=d+1$. If $E'$ is a subbundle of $Q_{k,d}$ corresponding to a subrepresentation as in (\ref{eq:subQ}), from the formula (\ref{eq:muF}) we obtain $\mu_F(E')= c_1(F)rk(E')-rk(F)c_1(E')\geq 0,$ with equality only if $i=d+1$. Hence $[Q_{k,d}]$ is $\mu$-stable.
\vskip1mm
\noindent  
Now we have to exclude the possibility that $Q_{k,d}\cong S_\mu V\otimes E $ with non trivial $\mu$. If such a isomorphism exists, we show first that then necessarily $S_\mu V=S^hV$ for some $h$. Indeed assume on the contrary that $\mu$ has more than one row. Then by Proposition 
\ref{prop:gradedschur} in $\mathcal{Q}(S_\mu V\otimes\OO)$ there would exist incomparable vertexes under the ordering described in Observation \ref{obs:quiverord}, for example all the vertexes $[S_\nu\Omega(|\mu|)]$ with $|\mu\setminus\nu|=1$. Given a vertex $E_i=[S_\alpha \Omega(l)]$ in $gr(E)$ choose a irreducible summand $G_\nu$ of $S_\nu\Omega(|\mu|)\otimes E_i$ for any  $\nu$ as above. Observe that $G_\nu$ is associated to some partition  $\nu'$ with $|\nu'|=|\nu|+|\alpha|=|\mu|+|\alpha|-1$ independent of $\nu$. Hence the vertexes $[G_\nu]$ of $\mathcal{Q}(S_\mu V\otimes E)$ would be incomparable, contradicting the known structure of the quiver associated to $Q_{k,d}$. Hence we are left with the possibility that $Q_{k,d}\cong S^h V\otimes E $ for some $h>0$ and homogeneous vector bundle $E$. If $gr(E)$ contains a component given by a line bundle, since we know that $gr(S^h V\otimes\OO)$ contains $\OO(h)$ as a component, then also in $gr(Q_{k,d})$ we would find a line bundle component. However this contradicts the fact that the only components of $gr(Q_{k,d})$ are $S^i\Omega(d)$ with $i\geq d+1>0$, which are not line bundles since $n\geq 2$. So we are left with $gr(E)$ without any line bundle component. Then any component of $gr(E)$ is of the form $[S_\alpha \Omega(l)]$ with $\alpha$ a non trivial partition. Consider also the component $S^h\Omega(h)$ of $gr(S^h V\otimes\OO)$. Then in $gr(Q_{k,d})$ we would find all those component coming from $S_\alpha \Omega(l)\otimes S^h\Omega(h)$, which have the form $S_\nu\Omega(l+h)$, where $\nu$ is any partitions obtained from $\alpha$ by adding $k$ boxed to it, with no two on the same column. Since $\alpha$ is non trivial one can easily see that there are at least two of such $\nu$, and they will produce incomparable elements in $\mathcal{Q}(Q_{k,d})$, a contradiction. 
 \end{proof}
\begin{question}{\em Note that in this way we have shown also the stability of the kernel of the sheaf epimorphism
$$\eta^k: S^kV\otimes \OO(d-k)\to S^d V\otimes\OO$$ for any $k> d\geq 0$. Then replacing $\eta^k$ with a sufficiently general map $\phi$ between the same sheaves one gets that $\ker(\phi)$ is also stable. This leads quite naturally to ask for numerical conditions on $\dim U$, $\dim W$, $n$ and $l$ so that a general map
$$\phi: U\otimes\OO_{\PP^n}(-l)\to W\otimes\OO_{\PP^n}$$ has stable kernel.
An answer to this question can be viewed as a non trivial extension of the results in \cite{coanda} and \cite{maro}, which provide conditions for the stability of $\ker(\phi)$ in the case $\dim W=1$.}\end{question}

\section{Global sections of $\P^k\OO(d)$ and Taylor truncation maps}\label{sec:thm2}
In this section we provide the proof of Theorem \ref{thm:main2} about the rank of the Taylor truncation map
$$ H^0\P^k\OO(d)\to H^0\P^h\OO(d),$$ which arises from the canonical sheaf epimorphism $\P^k\OO(d)\to \P^h\OO(d)$ defined for any $k\geq h$ and $d$.  From the preceding section we know $\P^k\OO(d)\cong S^kV\otimes\OO(d-k)$ for $d\geq k$ or $d<0$ and $\P^k\OO(d)\cong Q_{k,d}\oplus (S^dV\otimes \OO)$ for $0\leq d<k$. So $H^0\P^k\OO(d)$ is easily computed when $d\geq k$ or $d<0$, precisely we have $H^0\P^k\OO(d)=S^kV\otimes S^{d-k}V$ for $d\geq k$  and $H^0\P^k\OO(d)=0$ for $d<0$. 
This already shows that Theorem \ref{thm:main2} is trivially true for $d<0$. 
	To compute the global sections $H^0\P^k\OO(d)$ in the case $0\leq d<k$ we need to know the global sections of $Q_{k,d}$. For any $l$ we can compute the cohomology groups $H^0$ and $H^1$ of $S^{i}\Omega(l)$ using Bott's theorem, according to the formulation given in \cite{ko} theorem 5.14. Remind that by Bott's theorem only one of the cohomology groups $H^jS^{i}\Omega(l)$ can be non zero, with $j$ determined accordingly to the position of the maximal $P$-weight of $S^{i}\Omega(l)$ in the weight space. From the formula appearing in Proposition \ref{prop:graded} we know that the maximal $P$-weight associated to $S^{i}\Omega(l)$ is $\lambda=(l-i)L_1+iL_2=(l-2i)\lambda_1+i\lambda_2$. If $l\geq 2i$ then, setting $g=\sum_{i=1}^n$, one finds that $\lambda+g$ belongs to the interior of the Weyl chamber of $Sl(n+1)$ and then Bott's theorem says that $H^0S^{i}\Omega(l)=V_\lambda=S_{(l-i,i)} V$. We have therefore the following formula:
\begin{equation}\label{eq:h0symmomega} H^0S^{i}\Omega(l)=\left\{\begin{array}{ll}S_{(l-i,i)} V&\mbox{if }l\geq 2i\\
0&\mbox{otherwise}\end{array}\right.\end{equation}
Let us now assume $l<2i$  and let $w$ be the element of the Weyl group given by the reflection with respect to the hyperplane orthogonal to $\alpha_1$, and let us denote the Killing form for $Sl(n+1)$ with $(\cdot,\cdot)$, see \cite{fuha} for the notations. We have
\begin{eqnarray*} w(\lambda+g)&=&\lambda+g-2\frac{(\lambda+g,\alpha_1)}{(\alpha_1,\alpha_1)}\alpha_1\\
&=&\lambda+g- (l-2i+1)(2\lambda_1-\lambda_2)\\
&=&(2i-l-1)\lambda_1+(l-i+2)\lambda_2+\sum_{i=3}^n\lambda_i,\end{eqnarray*}
which belongs to the interior of the Weyl chamber of $Sl(n+1)$ iff $0\leq l-i+1<i$. 
In this case by Bott's theorem we have $H^1S^{i}\Omega(l)=V_\nu$ with $V_\nu$ the irreducible representation of $Sl(n+1)$ with maximal weight $\nu=w(\lambda+g)-g$, i.e. $\nu=(2i-l-2)\lambda_1+(l-i+1)\lambda_2=(i-1)L_1+(l-i+1)L_2$. Hence we obtain the following formula fo $H^1$:
\begin{equation}\label{eq:h1symmomega} H^1S^{i}\Omega(l)=\left\{\begin{array}{ll}S_{(i-1,l-i+1)} V&\mbox{if }0\leq l-i+1<i\\
0&\mbox{otherwise}\end{array}\right.\end{equation}
Applying the results above to the irreducible components of $Q_{k,d}$ we see that $H^0S^i\Omega(d)=0$ for any $i=d+1,\ldots,k$, and then, using the filtration from Proposition \ref{prop:quiverpp}, one sees easily that $H^0Q_{k,d}=0$. Hence one has the following result.
\begin{equation}\label{eq:h0pp}
H^0\P^k\OO(d)=\left\{\begin{array}{ll}S^kV\otimes S^{d-k}V&\mbox{for }d\geq k\\
S^dV & \mbox{for }0\leq d< k\\
0 & \mbox{for } d< 0\end{array}\right.\end{equation}
This allows to deal with one more case of Theorem \ref{thm:main2}, namely the case $h\leq d<k$. Infact, using the composition
$$\P^k\OO(d)\to \P^d\OO(d)\to\P^h\OO(d)$$ and the isomorphism $H^0\P^k\OO(d)\cong H^0\P^d\OO(d)=S^d V$ just proved, we can reduce the case $h\leq d<k$ to $h<d=k$. Therefore the proof of Theorem \ref{thm:main2} will be complete after treating the case $h< k\leq d$. 
From Proposition \ref{prop:dgeqkbis} one deduces the existence of the commutative diagram
\[\begin{CD} \P^k\OO(d) @>>>  \P^{h}\OO(d) \\ @V\cong V{u_k}V  @V\cong V {cu_{k-1}} V\\ S^k V\otimes \OO(d-k) @>\eta^{k-h}>> S^{h} V\otimes \OO(d-h)\end{CD}\] for a suitable non zero constant $c$. Using the results and the notations of the preceding section, we see that $$\ker\eta^{k-h}\cong Q_{k,h}(d-h).$$ Moreover, taking global sections we get the linear map 
\begin{equation}\label{eq:lineartrunc}\eta^{k-h}: S^kV\otimes S^{d-k}V\to S^hV\otimes S^{d-h}V,\end{equation} with the explicit formula  $$\eta^{k-h}=\sum_{|I|=k-h} {k-h\choose I}\partial^I\otimes x^I.$$
Recall that $S^iV^{\vee}$ can be identified with $S^i(\C\partial_0\oplus\cdots\oplus\C\partial_n)$ by identifying the dual basis of the monomials $(x^I)^{\ast}$ with the basis of differential operators $\partial^I/I!$, with $I!=I_0!\cdots I_n!$. In this setting one can write $\eta^{k-h}=(k-h)!\sum (x^I)^\ast\otimes x^I$. Knowing that multiplication by a homogeneous polynomials $F$, defined as a map $S^iV\stackrel{F}\to S^{i+\deg(F)}V$, is dual to the contraction $S^{i+\deg(F)}V^{\vee}\stackrel{\langle\cdot,F\rangle}\to S^iV^{\vee}$, we see that the map $\eta^{k-h}$ is self-adjoint under the linear isomorphism $Sym^\bullet V^{\vee}\cong Sym^\bullet V$ identifying $(x^I)^\ast$ with $x^I$. With this set up, we can reduce the proof that the map (\ref{eq:lineartrunc}) has maximal rank to the case of injectivity, i.e. to the following statement.
\begin{prop}[End of proof of Theorem \ref{thm:main2}]\label{prop:main2} Under the assumption that $\dim(S^kV\otimes S^{d-k}V)\leq\dim(S^hV\otimes S^{d-h}V)$ the map (\ref{eq:lineartrunc}) is injective.\end{prop}
\begin{proof} The dimensions of domain and codomain of the map (\ref{eq:lineartrunc}) are ${k+n\choose n}{d-k+n\choose n}$ \mbox{ and } ${h+n\choose n}{d-h+n\choose n},$ respectively. We claim that the inequality ${k+n\choose n}{d-k+n\choose n}\leq {h+n\choose n}{d-h+n\choose n}$ is equivalent to $d-k\leq h.$ Indeed we can rewrite it as
$ {d-k+n\choose n}/{d-k+k-h+n\choose n}\leq {h+n\choose n}/{h+k-h+n\choose n}$
and observe that the function ${x+n\choose n}/{x+k-h+n \choose n}=\prod_{i=1}^n \frac{x+i}{x+k-h+i}$ is increasing for $x\geq 0$. 
\noindent
The kernel (\ref{eq:lineartrunc}) is $H^0Q_{k,h}(d-h)$ so we aim to show that  $H^0Q_{k,h}(d-h)=0$ if $0\leq d-k\leq h$. Recall the filtration of $Q_{k,h}(d-h)$, known by Proposition \ref{prop:quiverpp}, with quotients $S^i\Omega(d)$ for $i=h+1,\ldots k$ and the associated quiver
$$[S^{h+1}\Omega(d)]\to\cdots\to [S^{k}\Omega(d)].$$ 
We will apply Theorem A of \cite{bora}, which in this context says that if for some $i=h+1,\ldots,k$ one has $H^0S^i\Omega(d)\cong U\not=0$ as a $G$-module and if there is also some  $j\leq k$ such that $H^1S^j\Omega(d)\cong U$ as a $G$-module, then $U$ does not appear in the decomposition of $H^0Q_{k,h}(d-h)$ as a $G$-module, otherwise it does. If $d<2(h+1)$ then all $H^0S^i\Omega(d)=0$ by (\ref{eq:h0symmomega}) and hence  one  has immediatedly $H^0Q_{k,h}(d-h)=0$. Otherwise we have $2(h+1)\leq d\leq h+k$ and define $r$ as the maximum such that $2r\leq d$. For any $i\geq 0$ with $h+1\leq i\leq r$ one has $H^0S^{i}\Omega(d)=S_{d-i,i} V$  by (\ref{eq:h0symmomega})  and $H^1S^{d-i+1}\Omega(d)=S_{d-i,i} V$ by (\ref{eq:h1symmomega}), so it remains to show that $d-i+1\leq k$. This follows from $d\leq h+k\leq i-1+k$. \end{proof}
\noindent
{\bf Remark.} Note that, because of the necessary and sufficient condition $d\leq h+k$ for injectivity given in the preceding proof, the Taylor truncation map $$H^0\P^k\OO(d)\to H^0\P^h\OO(d)$$ can be either injective or surjective and both cases can occur.


\begin{thebibliography}{1000}
\bibitem{bodeoli} F. Bogomolov and B. de Oliveira, {\em Tensors And Geometry of ${\mathbb{P}} ^N$ Subvarieties}, Geometric And Functional Analysis vol. 18, no 3 (2008), 637-656.
\bibitem{boka}  A.I. Bondal, M.M. Kapranov, {\em Homogeneous bundles} , in Helices and Vector Bundles, Proc. Semin. Rudakov, Lond. Math. Soc. Lect. Notes Ser. 148 (1990), 45-55. 
\bibitem{bora} A. Boralevi, {\em Sections of homogeneous vector bundles}, Journal of Algebra 323 (2010) 2301Ð2317.
\bibitem{coanda} I. Coanda, {\em On the stability of syzygy bundles}, Int. J. Math. 22, no. 4 (2011), 515-534 
\bibitem{ein} L. Ein, {\em Vanishing theorems for varieties of low codimension}, in Algebraic Geometry-Sundance 1986, Lext. Notes in Math. 1311, Berlin-Heidelberg-New Jork (1988). 
\bibitem{fuha} W. Fulton, J. Harris, Representation Theory, a First Course, Graduate Texts in Mathematics, Springer Verlag (1991).
\bibitem{gzk} I.M.Gelfand, M.M. Kapranov and A.V. Zelevinski,  Discriminant, Resultants, and Multidimensional Determinants, Mathematics: theory \& applications, Birkh\"auser, Boston (1994).
\bibitem{groth} A. Grothendieck, E.G.A. IV, 4\`eme partie, Publ. Math. I.H.E.S., tome 32 (1967),  5-361.
\bibitem{ise} M. Ise, {\em Some properties of complex analytic vector bundles over compact complex homogeneous spaces}, Osaka Math. J. 12 (1960), 217-252.
\bibitem{king} A. King, {\em Moduli of representations of finite dimensional algebras}, Q. J. Math., Oxf. 45, No.180 (1994), 515-530.
\bibitem{ko} B. Kostant, {\em Lie algebra cohomology and the generalized Borel-Weil theorem}, Ann. of Math. 74, (1961), 329-387.
\bibitem{maake1} H. Maakestad, {\em A note on the principal parts on projective space and linear representations},
Proc. of the AMS Vol. 133 no. 2 (2004).
\bibitem{maake2} H. Maakestad, {\em Jet bundles on projective spaces II}, arXiv:1003.0816.

\bibitem{maro} P. Macias Marques, R. M. Mir—-Roig, {\em Stability of syzygy bundles},
Proc. Amer. Math. Soc. 139 (2011), 3155-3170. 

\bibitem{ot} G. Ottaviani, Rational Homogeneous Varieties, Notes of SMI course held in Cortona (1995), available at http://www.math.unifi.it/ottavian.
\bibitem{otru} G. Ottaviani and E. Rubei, {\em Quivers and the cohomology of homogeneous vector bundles}, Duke Math. J. 132, no 3 (2006), 459-508. 
\bibitem{perk} D. Perkinson, {\em Principal parts of line bundles on toric varieties}, Comp. Math. tome 104 no 1 (1996),  27-39.
\bibitem{ram} S. Ramanan, {\em Holomorphic vector bundles on homogeneous spaces}, Topology 5 (1966), 159-177.
\bibitem{som} S. di Rocco and A.J. Sommese, {\em Line bundles for which a projectivized jet bundle is a product}, Proc. AMS vol. 129 no 6 (2000), 1659-1663.
\bibitem{samwey}S. V. Sam and J. Weyman, {\em Pieri resolutions for classical groups}, J. of Algebra, vol. 329, issue 1 (2011), 222-259.
\end{thebibliography}
\end{document}